\RequirePackage{fix-cm}
\documentclass[english]{article}
\usepackage[T1]{fontenc}
\usepackage[latin9]{inputenc}
\usepackage{geometry}
\usepackage{prettyref}
\usepackage{wrapfig}
\usepackage{amsthm}
\usepackage{amsmath}
\usepackage{amssymb}
\usepackage[all]{xy}

\makeatletter

\providecommand{\tabularnewline}{\\}

\numberwithin{equation}{section}
  \theoremstyle{plain}
  \newtheorem*{conjecture*}{\protect\conjecturename}
  \theoremstyle{plain}
  \newtheorem*{prop*}{\protect\propositionname}
\theoremstyle{plain}
\newtheorem{thm}{\protect\theoremname}[section]
  \theoremstyle{plain}
  \newtheorem{lem}[thm]{\protect\lemmaname}
  \theoremstyle{definition}
  \newtheorem{defn}[thm]{\protect\definitionname}
  \theoremstyle{remark}
  \newtheorem{rem}[thm]{\protect\remarkname}
  \theoremstyle{plain}
  \newtheorem{prop}[thm]{\protect\propositionname}
 \theoremstyle{definition}
 \newtheorem*{defn*}{\protect\definitionname}
  \theoremstyle{plain}
  \newtheorem{cor}[thm]{\protect\corollaryname}
  \theoremstyle{plain}
  \newtheorem*{lem*}{\protect\lemmaname}
  \theoremstyle{plain}
  \newtheorem{conjecture}[thm]{\protect\conjecturename}

\usepackage[perpage,para,symbol]{footmisc}
\usepackage[all]{xy}
\usepackage{cancel}
\usepackage[leftcaption]{sidecap}
\usepackage{graphicx}

\sidecaptionvpos{figure}{c}

\newcommand{\FigBesBeg}[1][1.0]{%
 \let\MyFigure\figure
 \let\MyEndfigure\endfigure
 \renewenvironment{figure}[1]{\begin{SCfigure}[#1]##1}{\end{SCfigure}}}

\newcommand{\FigBesEnd}{%
 \let\figure\MyFigure
 \let\endfigure\MyEndfigure}

\DefineFNsymbols*{lamportnostar}[math]{\dagger\ddagger\S\P\|{\dagger\dagger}{\ddagger\ddagger}}
\setfnsymbol{lamportnostar}

\newcommand{\Xcov}{{\scriptscriptstyle \overset{\twoheadrightarrow}{X}}}
\newcommand{\covers}{\leq_{\Xcov}}
\newcommand{\filleddiamond}{\raisebox{0.17\height}{\scalebox{0.9}[0.6]{$\blacklozenge$}}}

\makeatother

\usepackage{babel}
  \providecommand{\conjecturename}{Conjecture}
  \providecommand{\corollaryname}{Corollary}
  \providecommand{\definitionname}{Definition}
  \providecommand{\lemmaname}{Lemma}
  \providecommand{\propositionname}{Proposition}
  \providecommand{\remarkname}{Remark}
\providecommand{\theoremname}{Theorem}

\begin{document}

\title{Stallings Graphs, Algebraic Extensions\\
and Primitive Elements in $\mathbf{F}_{2}$}

\author{Ori Parzanchevski%
\thanks{Supported by an Advanced ERC Grant.%
} and Doron Puder%
\thanks{Supported by the Adams Fellowship Program of the Israel Academy of
Sciences and Humanities.%
}}
\maketitle
\begin{abstract}
This paper studies the free group of rank two from the point of view
of Stallings core graphs. The first half of the paper examines primitive
elements in this group, giving new and self-contained proofs for various
known results about them. In particular, this includes the classification
of bases of this group. The second half of the paper is devoted to
constructing a counterexample to a conjecture by Miasnikov, Ventura
and Weil, which seeks to characterize algebraic extensions in free
groups in terms of Stallings graphs.
\end{abstract}

\section{Introduction}

Let $\mathbf{F}$ be a finitely generated free group. A subgroup $J$
of $\mathbf{F}$ is said to be an \emph{algebraic extension} of another
subgroup $H$, if $H\leq J$ and there does not exist an intermediate
subgroup $H\leq M\lneq J$ such that $M$ is a proper free factor
of $J$. We denote this by $H\leq_{alg}J$\marginpar{$H\leq_{alg}J$}.
This notion, which was formulated independently by several authors
(and already appears in \cite{takahasi1951note}), is central to the
understanding of the lattice of subgroups of $\mathbf{F}$. For example,
it can be shown that every extension $H\leq J$ of free groups admits
a unique intermediate subgroup $H\leq_{alg}M\leq_{ff}J$ (where $\leq_{ff}$
denotes a free factor). Moreover, if $H\leq\mathbf{F}$ is a finitely
generated subgroup, it has only finitely many algebraic extensions
in $\mathbf{F}$. Thus, every group containing $H$ is a free extension
of one of the algebraic extensions of $H$, which is a well known
theorem of Takahasi \cite{takahasi1951note}. For proofs of the mentioned
facts, as well as a general survey of algebraic extensions, we refer
the reader to \cite{MVW07}.

Given a basis $X$ of $\mathbf{F}$ and $H\leq\mathbf{F}$, we denote
by $\Gamma_{X}\left(H\right)$\marginpar{$\Gamma_{X}\left(H\right)$}
the \emph{Stallings core graph }of $H$ with respect to $X$. This
is a pointed, directed, $X$-labeled graph, such that the words formed
by closed paths around the basepoint are precisely the elements of
$H$, and which is minimal with respect to this property. One way
to construct this graph is by taking the Schreier right coset graph
of $H$ in $\mathbf{F}$ w.r.t.\ $X$ and then deleting all {}``hanging
trees'', i.e., all edges which are not traced by some non-backtracking
loop around the basepoint. Figure \ref{fig:first_core_graph} demonstrates
the core graph of $H=\left\langle ab^{-1}a,a^{-2}b\right\rangle $
for $X=\left\{ a,b\right\} $ and $\mathbf{F}=\mathbf{F}\left(X\right)$.
We refer to \cite{Sta83,KM02,MVW07,Pud11} for further background
on Stallings graphs.

\FigBesBeg[1.7]
\begin{figure}[h]
\centering{}%
\begin{minipage}[t]{0.35\columnwidth}%
\[
\xymatrix@=15pt{\otimes\ar[rr]^{a} &  & \bullet\\
\\
\bullet\ar[rr]^{a}\ar[uu]_{b} &  & \bullet\ar[uull]_{a}\ar[uu]_{b}
}
\]
\end{minipage}\caption{\label{fig:first_core_graph} The core graph $\Gamma_{X}\left(H\right)$
where $X=\left\{ a,b\right\} $ and $H=\left\langle ab^{-1}a,a^{-2}b\right\rangle \leq\mathbf{F}\left(X\right)$.}
\end{figure}
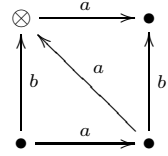

Given the basis $X$, and two subgroups $H,J\leq\mathbf{F}$, there
is a graph morphism (which preserves the basepoint, directions and
labeling) from $\Gamma_{X}\left(H\right)$ to $\Gamma_{X}\left(J\right)$
if and only if $H\leq J$. Such a morphism is unique, when it exists.
Given $H,J\leq\mathbf{F}$, we say that \emph{$H$ $X$-covers $J$}\marginpar{\emph{$X$-covers}}
if $H\leq J$ and the morphism from $\Gamma_{X}\left(H\right)$ to
$\Gamma_{X}\left(J\right)$ is onto. We denote this by $H\covers J$\marginpar{$\covers$}.
(In \cite{MVW07} this is indicated by saying that $J$ is a {}``$X$-principal
overgroup'' of $H$, and by the notation $J\in\mathcal{O}_{X}\left(H\right)$.)

\medskip{}

It is not hard to see (e.g.\ \cite[prop.\ 3.7]{MVW07}, or \cite[claim 3.2]{puder2012measure})
that if $H\leq_{alg}J$, then $H\covers J$ for every basis $X$ of
$\mathbf{F}$. The following conjecture, raised in \cite{MVW07},
asks whether the converse also holds.
\begin{conjecture*}[{\cite[§5(1)]{MVW07}}]
If $H\leq J\leq\mathbf{F}$ and $H\covers J$ for every basis $X$
of $\mathbf{F}$ then $J$ is an algebraic extension of $H$.
\end{conjecture*}
The main result of this paper is a counterexample to this conjecture:
\begin{prop*}[Prop.\ \ref{prop:counterexample}]
Let $\mathbf{F}_{2}=\mathbf{F}\left(a,b\right)$ be the free group
on two generators, $H=\left\langle a^{2}b^{2}\right\rangle $, and
$J=\left\langle a^{2}b^{2},ab\right\rangle $. Then $H\covers J$
for every basis $X$ of $\mathbf{F}_{2}$, but $J$ is not an algebraic
extension of $H$.
\end{prop*}
The relation {}``$H\covers J$'' is basis-dependent, while the relation
{}``$H\covers J$ for every basis $X$'' is intrinsic, as is {}``$H\leq_{alg}J$''.
Proposition \ref{prop:counterexample} means that the latter two relations
are different, and this raises the intriguing question of understanding
the algebraic significance of {}``covering with respect to all bases''.

\medskip{}
The proof of Proposition \ref{prop:counterexample} follows from a
thorough analysis of Stallings graphs, using classical results (e.g.\ \cite{nielsen1917isomorphismen,cohn1972markoff,cohen1981does,OZ81})
on primitive elements and bases of $\mathbf{F}_{2}$. It turns out
that these results can also be proven by appealing solely to Stallings
graphs, and we use the opportunity to provide self-contained proofs
for them in Section \ref{sec:Primitives-in-f_2}. Section \ref{sec:Stallings-Graphs}
recalls some basic facts about Stallings graphs and foldings, and
presents two auxiliary lemmas which will be used later on. Finally,
the proof of the counterexample (Proposition \ref{prop:counterexample})
is given in Section \ref{sec:The-counterexample}, and some concluding
remarks in Section \ref{sec:Epilogue}.

\section{Stallings Graphs\label{sec:Stallings-Graphs}}

We assume that the reader is familiar with the theory of Stallings
foldings, but recall the basic facts. If $\Gamma$ is a pointed, directed,
$X$-labeled graph, we denote by $\pi_{1}^{X}\left(\Gamma\right)$\marginpar{$\pi_{1}^{X}\left(\Gamma\right)$}
the subgroup of $\mathbf{F}=\mathbf{F}\left(X\right)$ consisting
of the words which appear as closed loops around the basepoint of
$\Gamma$. The operators $\pi_{1}^{X}$ and $\Gamma_{X}$ constitute
a bijection between subgroups of $\mathbf{F}\left(X\right)$ and $X$-labeled
core graphs, which matches f.g.\ subgroups to finite graphs.

If $\Gamma$ is a finite (pointed, directed) $X$-labeled graph, and
$\pi_{1}^{X}\left(\Gamma\right)=H$, then $\Gamma_{X}\left(H\right)$
is obtained from $\Gamma$ by repeatedly performing one of the following
operations, in any order, until neither of them is possible:
\begin{enumerate}
\item \emph{Folding} - merging two edges with the same label, and the same
origin or terminus (and thus merging also the other ends).
\item \emph{Trimming} - deleting a leaf which is not the basepoint, and
the edge which leads to it.
\end{enumerate}
The following lemma shows that under certain conditions only foldings
are necessary in this process:
\begin{lem}
\label{lem:no-trim}Let $\Gamma$ be a finite, pointed, directed,
$X$-labeled graph such that at every vertex, except possibly the
basepoint, there are at least two types of edges (the type of an edge
consists of its label and direction). Then the core graph $\Gamma_{X}\left(H\right)$
of $H=\pi_{1}^{X}\left(\Gamma\right)$ is obtained from $\Gamma$
by foldings alone (i.e.\ without trimming).\end{lem}
\begin{proof}
Evidently, $\Gamma$ cannot have leaves, except for possibly the basepoint.
Folding steps do not decrease the number of types of edges at a vertex,
so that the property in the statement still holds after every folding
step, and no new leaves are created throughout the process.
\end{proof}
This simple lemma will prove out to be extremely useful. It already
plays a role in Lemma \ref{lem:covers-iff-appears}, which characterizes
$X$-covering in simple extensions.
\begin{defn}
\label{def:appear}Let $\Gamma$ be a pointed and directed $X$-labeled
graph and let $w\in\mathbf{F}$. We say that \emph{$w$ appears}\marginpar{\emph{appears in}}\emph{
in $\Gamma$} if there exist paths $p_{1},p_{2}$ in $\Gamma$ such
that $p_{1}$ starts at the basepoint, $p_{2}$ terminates at the
basepoint, and $w=p_{1}p_{2}$ (i.e.\ $p_{1}p_{2}$ is the presentation
of $w$ as a reduced word in $X\bigcup X^{-1}$).
\end{defn}
For example, for $H$ in Figure \ref{fig:first_core_graph}, $a^{3}$
and $a^{2}ba^{-1}$ appear in $\Gamma_{X}\left(H\right)$, but $a^{2}b^{2}$
does not. Notice that if $w$ appears in $\Gamma$, s.t.\ $\pi_{1}^{X}\left(\Gamma\right)=H$,
and $\Gamma$ satisfies the conditions of Lemma \ref{lem:no-trim},
then $w$ appears in $\Gamma_{X}\left(H\right)$ as well. This will
play a significant part in Section \ref{sec:The-counterexample}.
\begin{lem}
\label{lem:covers-iff-appears}Let $H\leq\mathbf{F}$, $w\in\mathbf{F}$
and $J=\left\langle H,w\right\rangle $. Then $H\covers J$ iff $w$
appears in $\Gamma_{X}\left(H\right)$.
\end{lem}
\textit{Proof.} Assume first that $w$ appears in $\Gamma_{X}\left(H\right)$,
and let $p_{1},p_{2}$ be as in Definition \ref{def:appear}. Denote
by $\Gamma$ the graph obtained from $\Gamma_{X}\left(H\right)$ by
identification of $p_{1}$'s endpoint and $p_{2}$'s start-point.
We have $\pi_{1}^{X}\left(\Gamma\right)=J$, and the (pointed, directed,
labeled) map from $\Gamma_{X}\left(H\right)$ to $\Gamma$ is onto.
Since $\Gamma$ satisfies the conditions of Lemma \ref{lem:no-trim},
$\Gamma_{X}\left(J\right)$ is obtained from it by foldings alone.
We have now that $\Gamma_{X}\left(H\right)$ maps onto $\Gamma$,
which maps onto $\Gamma_{X}\left(J\right)$, and by transitivity it
follows that $H\covers J$.

\begin{wrapfigure}{r}{0.4\columnwidth}%
\xy(0,0)*{\otimes}="bp"+(-6,8)*{\Gamma_{X}\left(H\right)};	 (5,0)*\xycircle(20,20){--}; (22.32,10)*{\bullet}="s1"  +(4,4)*{v_1}; (22.32,-10)*{\bullet}="s2" +(4,-4)*{v_2}; "bp"; "s1" **\crv{(10,20) & (15,-20)} ?(.43)*\dir{>>}; (6,8)*{\scriptstyle p_{1}}; "s2"; "bp" **\crv{(20,20) & (5,-10)} ?(.33)*\dir{>>}; (6,-4)*{\scriptstyle p_2}; "s1";"s2" **\crv{(40,10) & (40,-10)} ?(.5)*\dir{>>}; (40,0)*{\scriptstyle p_{w'}}; \endxy \caption{\label{fig:handle}$\Gamma=\Gamma_{X}\left(H\right)\bigcup p_{w'}$}
\end{wrapfigure}%

Assume now that $w$ does not appear in $\Gamma_{X}\left(H\right)$.
Let $p_{1}$ be the maximal path beginning at the basepoint of $\Gamma_{X}\left(H\right)$
which is a prefix of $w$, and denote by $v_{1}$ its endpoint. Let
$p_{2}$ be the maximal path ending at the basepoint of $\Gamma_{X}\left(H\right)$
which is a suffix of $w$, and $v_{2}$ its beginning. If $w=p_{1}w'p_{2}$,
take $\Gamma=\Gamma_{X}\left(H\right)\bigcup p_{w'}$ where $p_{w'}$
is a path labeled by $w'$, whose beginning is attached to $v_{1}$,
and whose endpoint to $v_{2}$ (see Figure \ref{fig:handle}). Now
$\pi_{1}^{X}\left(\Gamma\right)=J$ and $\Gamma$ has no foldable
edges nor leaves, i.e.\ $\Gamma=\Gamma_{X}\left(J\right)$. Thus
$\Gamma_{X}\left(H\right)$ is a subgraph of $\Gamma_{X}\left(J\right)$,
and in particular does not map onto it. (In fact, since the map from
$\Gamma_{X}\left(H\right)$ to $\Gamma_{X}\left(J\right)$ is injective,
$H$ is a free factor of $J$.)\qed
\begin{rem}
With some further work, the basic idea of Lemma \ref{lem:covers-iff-appears}
can lead to an algorithm to detect primitive words and free factors
in $\mathbf{F}$. See \cite[Thm 1]{Pud11}.
\end{rem}

\section{\label{sec:Primitives-in-f_2}Primitives in $\mathbf{F}_{2}$}

In this section we give new proofs for the classical theorems on primitive
words and bases of $\mathbf{F}_{2}$ \cite{nielsen1917isomorphismen,cohn1972markoff,cohen1981does,OZ81}.
Throughout the section $X$ denotes the basis $\left\{ a,b\right\} $
of $\mathbf{F}_{2}=\mathbf{F}\left(a,b\right)$.

We start with the following lemma, which reduces the classification
of bases of $\mathbf{F}_{2}$ to that of cyclically reduced (henceforth:
CR)\marginpar{\emph{CR}} bases.
\begin{lem}
\label{lem:general-bases}Let $Y=\left\{ \overline{u},\overline{v}\right\} $
be any basis of $\mathbf{F}_{2}$.
\begin{enumerate}
\item Write $\overline{u}=xux^{-1}$ %
\footnote{By {}``write'' we mean that $xux^{-1}$ is a reduced expression
of $\overline{u}$ - no cancellation is needed. This convention will
repeat throughout the paper, and we will not mention it again.%
} and $\overline{v}=yvy^{-1}$ with $u,v$ CR. Then either $x$ is
a prefix of $y$ or $y$ is a prefix of $x$.
\item Assume that $x$ is a prefix of $y$, and write $\overline{u}=xux^{-1}$
and $\overline{v}=xwvw^{-1}x^{-1}$. Then $w$ is a prefix of some
power of $ $$u$ or of $u^{-1}$ (which implies that $w^{-1}uw$
is a cyclic rotation of $u$).
\item The basis $\left(xw\right)^{-1}Yxw=\left\{ w^{-1}uw,v\right\} $ is
CR.
\end{enumerate}

Therefore, any basis of $\mathbf{F}_{2}$ is of the form $\left\{ xux^{-1},xwvw^{-1}x^{-1}\right\} $
where $w$ is a prefix of some power of $u^{\pm1}$, $\left\{ w^{-1}uw,v\right\} $
is a CR basis, and $x$ is any word s.t.\ $xux^{-1}$ and $xwvw^{-1}x^{-1}$
are reduced.

\end{lem}
\begin{proof}
The graph $\Gamma=\xymatrix@1@C=15pt{\bullet\ar@(dl,ul)[]^{v}&\otimes\ar[r]^x\ar[l]_y&\bullet\ar@(dr,ur)[]_{u}}$
satisfies $\pi_{1}^{X}\left(\Gamma\right)=\left\langle \overline{u},\overline{v}\right\rangle =\mathbf{F}_{2}$.
It also satisfies the conditions of Lemma \ref{lem:no-trim}, and
must therefore fold into $\Gamma_{X}\left(\mathbf{F}_{2}\right)=\vphantom{\Big|}\xymatrix@1{\otimes\ar@(dl,ul)[]^{a}\ar@(dr,ur)[]_{b}}$.
The only vertex at which folding may occur is the basepoint $\otimes$,
and after the first identification of edges the only possible folding
place is at the identified termini. Continuing in this manner shows
that for $\Gamma$ to fold into $\Gamma_{X}\left(\mathbf{F}_{2}\right)$,
the shorter word among $x,y$ must be completely merged with a prefix
of the longer one, giving $\left(1\right)$. By the same arguments,
the graph $\Gamma'=\xymatrix@1@C=15pt{\bullet\ar@(dl,ul)[]^{v}&\otimes\ar@(dr,ur)[]_{u}\ar[l]_w}$
must fold into $\Gamma_{X}\left(\mathbf{F}_{2}\right)$, and for this
to happen $w$ must wind itself completely around $u$, or around
$u^{-1}$ (i.e.\ $w$ must be a prefix of some power of $u$ or of
$u^{-1}$). It follows that $w^{-1}uw$ is a cyclic rotation of $u$,
and $\left\{ w^{-1}uw,v\right\} $ is a CR basis.
\end{proof}
Moving on to CR bases, we have the following:
\begin{prop}
\label{prop:basis-prefix}Let $\left\{ u,v\right\} $ be a CR basis
of $\mathbf{\mathbf{F}}_{2}$, such that $\left|u\right|+\left|v\right|\geq3$\ %
\footnote{Here $\left|w\right|$ is the length of $w$ as a reduced word in
$X\cup X^{-1}$.%
}, and $\left|u\right|\leq\left|v\right|$. Then either $u$ is a prefix
or a suffix of $v$, or $u^{-1}$ is.\end{prop}
\begin{proof}
Since $\left|v\right|\geq2$ and $v$ is not a proper power, $v$
contains both $a$ and $b$, and thus $\left\langle v\right\rangle \covers\mathbf{\mathbf{F}}_{2}$.
Since $\mathbf{F}_{2}=\left\langle u,v\right\rangle $, Lemma \ref{lem:covers-iff-appears}
implies that $u$ appears in $\Gamma_{X}\left(\left\langle v\right\rangle \right)$,
which is just a cycle labeled by $v$ as a word in $X\cup X^{-1}$.

Let $p'$ be the maximal prefix of $u$ which is a path emanating
from the basepoint of $\Gamma_{X}\left(\left\langle v\right\rangle \right)$.
Since $\left|u\right|\leq\left|v\right|$, this means that $p'$ is
a prefix of $v$ or of $v^{-1}$. By inverting $v$ if necessary we
assume that $p'$ is a prefix of $v$. Let $s'$ be the maximal suffix
of $u$ which is a path ending at the basepoint of $\Gamma_{X}\left(\left\langle v\right\rangle \right)$.
Since $u$ is CR, $s'$ cannot be a suffix of $v^{-1}$, and must
be a suffix of $v$. 

Let $m$ be the middle part of $u$ where $p'$ and $s'$ overlap
(it may be empty: $\left|m\right|=\left|p'\right|+\left|s'\right|-\left|u\right|\geq0$).
Write $p'=pm$, $s'=ms$, which means that $u=pms$ (see Figure \ref{fig:v-decomp}).
Thus, if $p$ is empty then $u=pms=s'$ is a suffix of $v$, and if
$s$ is empty then $u$ is a prefix of $v$. We proceed to show that
they cannot be both nonempty.

\begin{figure}[h]
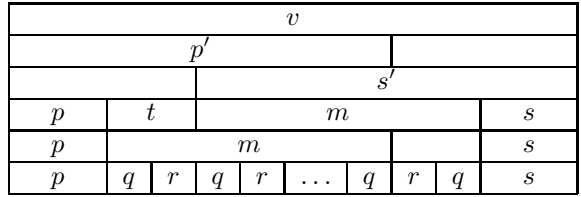

\noindent \centering{}%
\begin{tabular}{|c|c|c|c|c|c|c|c|c|c|c}
\cline{1-10} 
\multicolumn{10}{|c|}{$v$} & \tabularnewline
\cline{1-10} 
\multicolumn{7}{|c|}{$p'$} & \multicolumn{3}{c|}{} & \tabularnewline
\cline{1-10} 
\multicolumn{3}{|c|}{} & \multicolumn{7}{c|}{$s'$} & \tabularnewline
\cline{1-10} 
\quad{}$p$\quad{} & \multicolumn{2}{c|}{$t$} & \multicolumn{6}{c|}{$m$} & \quad{}$s$\quad{} & \tabularnewline
\cline{1-10} 
\quad{}$p$\quad{} & \multicolumn{6}{c|}{$m$} & \multicolumn{2}{c|}{} & \quad{}$s$\quad{} & \tabularnewline
\cline{1-10} 
\quad{}$p$\quad{} & $q$ & $r$ & $q$ & $r$ & $\ldots$ & $q$ & $r$ & $q$ & \quad{}$s$\quad{} & \tabularnewline
\cline{1-10} 
\end{tabular}\caption{\label{fig:v-decomp}Illustration of the decomposition of $v$.}
\end{figure}

Let $\Gamma=\xymatrix@1{\otimes\ar@(dl,ul)[]^{u}\ar@(dr,ur)[]_{v}}$.
Since $\pi_{1}^{X}\left(\Gamma\right)=\left\langle u,v\right\rangle =\mathbf{F}_{2}$
and $\Gamma$ satisfies the conditions of Lemma \ref{lem:no-trim},
it must fold into $\Gamma_{X}\left(\mathbf{F}_{2}\right)=\vphantom{\Big|}\xymatrix@1{\otimes\ar@(dl,ul)[]^{a}\ar@(dr,ur)[]_{b}}$,
and we will show that this cannot happen if $p,s\neq1$.

Assume therefore that $p,s\neq1$. Since $\left|v\right|\geq\left|u\right|$
we can write $v=ptms$, and $t\neq1$ since otherwise $u=v$ (this
shows, in particular, that $\left|v\right|>\left|u\right|$). Since
$pm$ is a prefix of $v$, $m$ is a prefix of $tm$, which means
that $m$ is a prefix of some (positive) power of $t$ (see again
Figure \ref{fig:v-decomp}). We consider two cases:

\medskip{}

\textbf{\emph{Case }}\textbf{$\boldsymbol{\left(i\right)}$}: $m$
is not a power of $t$. In this case $t=qr$ with $q,r\neq1$, and
$m=\left(qr\right)^{n}q$ with $n\geq0$ (see Figure \ref{fig:v-decomp};
$n=0$ corresponds to the possibility that $p'$ and $s'$ do not
overlap in $v$). Since $u=p\left(qr\right)^{n}qs$ and $v=p\left(qr\right)^{n+1}qs$,
$\Gamma$ folds into $\Gamma'=\vcenter{\xymatrix@1@R=1.5pt@C=15pt{ & \bullet\ar[dd]_{q}\\ \otimes\ar[ru]^{p}\\  & \bullet\ar[lu]^{s}\ar@/_{1pc}/[uu]^{r} } }$.

We aim to show that no folding can occur in $\Gamma'$, but let us
first introduce the following notations: for $w\in\mathbf{F}\left(X\right)$,
we denote by $w_{1}$ the first letter of $w$ as a reduced word in
$X\cup X^{-1}$, and for two words $w,w'$ we write $w\perp w'$ \marginpar{$\perp$}to
indicate that $w_{1}\neq\left(w'\right)_{1}$. Namely, $w\perp w'$
implies that no folding occurs in $\smash{\xymatrix@1{{}&\bullet\ar[l]_{w}\ar[r]^{w'}&{}}}$.

Returning to $\Gamma'$, we have $p\perp s^{-1}$ since $u$ (or equivalently
$v$) is CR, so no folding occurs at $\otimes$. Since $v=p\left(qr\right)^{n+1}qs$
is reduced, $p^{-1},r^{-1}\perp q$ and $q^{-1}\perp r,s$. We also
have $r\perp s$, for otherwise $p's_{1}=p\left(qr\right)^{n}qr_{1}$
would be a common prefix of $u=p's$ and $v=p\left(qr\right)^{n+1}qs$,
contradicting the maximality of $p'$. Finally, $r^{-1}\perp p^{-1}$
follows in the same way from the maximality of $s'$, and we conclude
that $\Gamma'$ cannot be folded any further, i.e.\ $\Gamma'=\Gamma_{X}\left(\left\langle u,v\right\rangle \right)$,
which contradicts $\Gamma_{X}\left(\left\langle u,v\right\rangle \right)=\Gamma_{X}\left(\mathbf{F}_{2}\right)$.

\textbf{\emph{Case }}\textbf{$\boldsymbol{\left(ii\right)}$}: $m$
equals a power of $t$, $m=t^{n}$ ($n\geq0$), so that $u=pt^{n}s$
and $v=pt^{n+1}s$. This time $\Gamma$ folds into $\Gamma'=\xymatrix@1{\otimes\ar@/^{0.7pc}/[r]_{p}&\bullet\ar@(ur,dr)[]_{t}\ar@/^{0.7pc}/[l]_s}$.
We have $p\perp s^{-1}$ as before; $p^{-1}\perp t$ and $t^{-1}\perp s$
follows from $v=pt^{n+1}s$ being reduced; $s\perp t$ holds, since
otherwise $pt^{n}s_{1}=pt^{n}t_{1}$ would be a common prefix of $u$
and $v$, contradicting the maximality of $p'=pt^{n}$; likewise,
$p^{-1}\perp t^{-1}$ by the maximality of $s'$. Now, if $n>0$ then
$t\perp t^{-1}$ since $v=pt^{n+1}s$ is reduced, and if $n=0$ then
$p^{-1}\perp s$ since $u=ps$ is reduced. In either case, $\Gamma'$
cannot fold into $\Gamma_{X}\left(\mathbf{F}_{2}\right)$: For $n=0$,
assuming that $\Gamma'$ folds at all, $\Gamma_{X}\left(\left\langle u,v\right\rangle \right)=\xymatrix@1{\otimes\ar@/^{0.7pc}/[r]_{p}&\bullet\ar@/^{0.7pc}/[l]_{s}\ar[r]^{r}&\bullet\ar@(ur,dr)[]_{t'}}$,
where $t=rt'r^{-1}$ with $t'$ CR. Thus, $\Gamma_{X}\left(\left\langle u,v\right\rangle \right)\neq\Gamma_{X}\left(\mathbf{F}_{2}\right)$.
For $n>0$, $\Gamma'$ folds into $\xymatrix@1{\otimes\ar@/^{0.9pc}/[r]_{\smash{p'}\vphantom{P}}&\bullet\ar@/^{0.9pc}/[l]_{\smash{s'}}\ar[r]^{r}&\bullet\ar@(ur,dr)[]_{t}}$
where $r$ is the common suffix of $p$ and $s^{-1}$, so that $p=p'r$,
$s=s'r$. If $p',s'\neq1$ we are done, and $p'=s'=1$ is impossible
since $p\perp s^{-1}$. If $p'=1\neq s'$ then the graph folds into
$\xymatrix@1@C=15pt{\bullet\ar@(dl,ul)[]_{s''}&\otimes\ar[r]^r\ar[l]_y&\bullet\ar@(dr,ur)[]^{t}}$
where $s''$ is CR and $s'=ys''y^{-1}$, and likewise for $p'\neq1=s'$.\end{proof}
\begin{defn*}
A word $w\in\mathbf{F}\left(a,b\right)$ is \emph{monotone}\marginpar{\emph{monotone}}\emph{
}if for every letter ($a$ or $b$) all the exponents of this letter
in $w$ have the same sign. \end{defn*}
\begin{prop}
\label{prop:monotone}A CR primitive word in $\mathbf{F}_{2}$ is
monotone.\end{prop}
\begin{proof}
Let $u$ be a CR primitive. By Lemma \ref{lem:general-bases}, possibly
applying some cyclic rotation to $u$, we can complete it to a CR
basis $\left\{ u,v\right\} $. We show that both $u$ and $v$ are
monotone, by induction on $\left|u\right|+\left|v\right|$. The base
case $\left|u\right|+\left|v\right|=2$ is trivial. Assume that $\left|u\right|\leq\left|v\right|$.
Using Proposition \ref{prop:basis-prefix}, and perhaps replacing
$u$, $v$, or both of them by their inverses (which does not affect
monotonicity), we can write $v=ut$. Now $\left\{ u,t\right\} $ is
a basis with $\left|u\right|+\left|t\right|<\left|u\right|+\left|v\right|$,
and we claim that $t$ is CR as well. Otherwise, $t=rt'r^{-1}$ with
$t'$ CR and $r\neq1$, and we have $\Gamma_{X}\left(\left\langle u,t\right\rangle \right)=\xymatrix@1@C=15pt{\otimes\ar[r]^{r}\ar@(dl,ul)[]^{u} & \bullet\ar@(dr,ur)[]_{t'}}$,
as $u\perp u^{-1}$ since $u$ is CR, $\left(t'\right)^{-1}\perp t'$
since $t'$ is, and all other relevant pairs since $v=urt'r^{-1}$
is. This, of course, contradicts $\left\langle u,t\right\rangle =\mathbf{F}_{2}$.

Therefore, by the induction hypothesis $u$ and $t$ are monotone.
Assume first that $\left|u\right|\leq\left|t\right|$. Since $v=ut$
is CR, $u^{-1}$ cannot be neither a prefix nor a suffix of $t$.
Thus, by Proposition \ref{prop:basis-prefix} $u$ must be a prefix
or a suffix of $t$, and in either case $v$ is monotone. The same
argument applies to the case $\left|u\right|>\left|t\right|$.
\end{proof}
We stress the following observation made in the proof:
\begin{cor}
\label{cor:shortening-CR-basis}Let $\left\{ u,v\right\} $ be a CR
basis of $\mathbf{F}_{2}$ with $u$ a prefix of $v$, and write $v=ut$.
Then $\left\{ u,t\right\} $ is again a CR basis.
\end{cor}
This leads to a constructive description of all CR bases of $\mathbf{F}_{2}$:
\begin{prop}
\label{prop:procedure-for-CR-bases}Any CR basis of $\mathbf{F}_{2}$
is obtained as follows: given a pair of positive co-prime integers
$\left(p,q\right)$, there is a unique sequence of pairs 
\begin{equation}
\left(p,q\right)=\left(p_{0},q_{0}\right),\left(p_{1},q_{1}\right),\ldots,\left(p_{\ell},q_{\ell}\right)=\left(1,1\right)\label{eq:euclides-numbers}
\end{equation}
which is the result of applying the Euclidean g.c.d.\ algorithm (i.e.\ if
$p_{i}<q_{i}$ then $p_{i+1}=p_{i}$ and $q_{i+1}=q_{i}-p_{i}$, and
vice-versa). Let $X_{\ell}=\left\{ u_{\ell},v_{\ell}\right\} $ be
one of the four bases $\left\{ a^{\pm1},b^{\pm1}\right\} $, and define
$X_{i}=\left\{ u_{i},v_{i}\right\} $ iteratively for $i=\ell-1\ldots0$
by 
\begin{equation}
\left(u_{i},v_{i}\right)=\begin{cases}
\left(u_{i+1}\hphantom{v_{i+1}}\:,\: v_{i+1}u_{i+1}\right) & p_{i}<q_{i}\\
\left(u_{i+1}v_{i+1}\:,\: v_{i+1}\hphantom{u_{i+1}}\right) & q_{i}<p_{i}.
\end{cases}\label{eq:straight-euclides}
\end{equation}
Finally, take $X_{0}$, conjugate its elements by any common prefix
or suffix (thus cyclically rotating both of them), and possibly replace
one of them by its inverse.\end{prop}
\begin{proof}
This construction certainly gives a CR basis of $\mathbf{F}_{2}$
(CR follows from monotonicity), and it remains to show that every
CR basis is thus obtained. This is done by reversing the process,
as follows. Let $\left\{ x_{0},y_{0}\right\} $ be a CR basis. Discarding
the trivial bases $\left\{ a^{\pm1},b^{\pm1}\right\} $, we can assume
(by inversion if necessary) that one of $x_{0}$ or $y_{0}$ is a
prefix/suffix of the other (by Proposition \ref{prop:basis-prefix}). 
\begin{lem*}
$x_{0},y_{0}$ can be rotated by a a common prefix/suffix, so that
a sequence of CR bases $\left\{ x_{1},y_{1}\right\} ,\ldots,\left\{ x_{\ell},y_{\ell}\right\} $
with the following properties is obtained:
\begin{enumerate}
\item For every $0\leq i\leq\ell-1$, the shorter of $x_{i},y_{i}$ is a
suffix of the longer one.
\item Each basis is obtained from the previous one by
\begin{equation}
\left(x_{i+1},y_{i+1}\right)=\begin{cases}
\left(x_{i},t\right) & \left|x_{i}\right|<\left|y_{i}\right|\:\mathrm{and}\: y_{i}=tx_{i}\\
\left(t,y_{i}\right) & \left|y_{i}\right|<\left|x_{i}\right|\:\mathrm{and}\: x_{i}=ty_{i}.
\end{cases}\label{eq:reverse-euclides}
\end{equation}

\item $\left\{ x_{\ell},y_{\ell}\right\} $ is one of the four bases $\left\{ a^{\pm1},b^{\pm1}\right\} $.\end{enumerate}
\begin{proof}[Proof of the Lemma]
If we do not perform the rotation of $x_{0},y_{0}$ by a common prefix/suffix,
the same holds, but with the exception that at each stage the shorter
among $x_{i},y_{i}$ may be a prefix of the longer one, and not a
suffix: this follows from by Proposition \ref{prop:basis-prefix},
and Corollary \ref{cor:shortening-CR-basis}, which ensures that all
of these bases are CR. Assume that the process first fails at step
$m$, i.e.\ the shorter among $x_{m},y_{m}$ is not a suffix of the
longer one, and assume for simplicity that $x_{m}$ is shorter, so
that $y_{m}=x_{m}t$. 

Since $x_{0},y_{0}$ are products of positive powers of $x_{m}$ and
$y_{m}=x_{m}t$, it follows that $x_{m}$ is a common prefix of both
of them. Therefore, $\left\{ x_{0}',y_{0}'\right\} =\left\{ x_{m}^{-1}x_{0}x_{m},x_{m}^{-1}y_{0}x_{m}\right\} $
is a cyclic rotation of $\left\{ x_{0},y_{0}\right\} $. Let $\left\{ x_{i}',y_{i}'\right\} _{i=0\ldots m}$
be the bases obtained by \prettyref{eq:reverse-euclides} from $\left\{ x_{0}',y_{0}'\right\} $.
Since the expression of $x_{i}',y_{i}'$ as words in $x_{m}',y_{m}'$
is the same as that of $x_{i},y_{i}$ as words in $x_{m},y_{m}$,
we still have that at every step until $m$ the shorter of $x_{i}',y_{i}'$
is a suffix of the longer one. In fact, $x_{i}'=x_{m}^{-1}x_{i}x_{m}$
and likewise for $y_{i}'$, for all $i\leq m$. Now, assertion \emph{(1)}
holds for step $m$ as well, as $x'_{m}=x_{m}$ and $y'_{m}=x_{m}^{-1}\left(x_{m}t\right)x_{m}=tx_{m}=tx_{m}'$. 

We continue in this manner: at the next step at which assertion \emph{(1)}
fails we replace $\left\{ x_{0}',y_{0}'\right\} $ by $\left\{ x_{0}'',y_{0}''\right\} $
which resolves that step, and note that $x_{0}'',y_{0}''$ is still
a cyclic rotation of the original $x_{0},y_{0}$ by a common prefix/suffix,
and that no new failures of \emph{(1)} were introduced by this change
for the previous steps. Repeating this for every failure of \emph{(1)}
guarantees that it hold throughout the process.
\end{proof}
\end{lem*}
We continue the proof of the proposition, assuming that $x_{0},y_{0}$
were inverted and rotated according to the Lemma, and $\left\{ x_{1},y_{1}\right\} ,\ldots,\left\{ x_{\ell},y_{\ell}\right\} $
are the bases obtained by \prettyref{eq:reverse-euclides}. The sequence
of integer pairs $\left(\left|x_{0}\right|,\left|y_{0}\right|\right),\ldots,\left(\left|x_{\ell}\right|,\left|y_{\ell}\right|\right)=\left(1,1\right)$
is then the sequence obtained by the Euclidean algorithm for $\left(\left|x_{0}\right|,\left|y_{0}\right|\right)$
(as in \prettyref{eq:euclides-numbers}), and in particular this shows
that $\left|x_{0}\right|,\left|y_{0}\right|$ are co-prime. Thus,
if one takes $\left(p,q\right)=\left(\left|x_{0}\right|,\left|y_{0}\right|\right)$
and $\left(u_{\ell},v_{\ell}\right)=\left(x_{\ell},y_{\ell}\right)$,
and follows \prettyref{eq:straight-euclides} as explained in the
statement of the proposition, the process in \prettyref{eq:reverse-euclides}
is reversed, and one obtains $\left(u_{0},v_{0}\right)=\left(x_{0},y_{0}\right)$.\end{proof}
\begin{cor}
For a CR basis $\left\{ u,v\right\} $, regard $u$ and $v$ as cyclic
words, and assume (by inverting $u$ if necessary) that one of them
is a subword of the other. Then one of $a,b$ always appears (in both
$u$ and $v$) with exponent $\varepsilon$ for some fixed $\varepsilon\in\left\{ 1,-1\right\} $,
and the other letter always appears with exponent $m$ or $m+1$ for
some $m\in\mathbb{Z}$.\end{cor}
\begin{proof}
Let $X_{i}=\left\{ u_{i},v_{i}\right\} $ ($0\leq i\leq\ell$) be
the bases constructed in Proposition \prettyref{prop:procedure-for-CR-bases}
to give $X_{0}=\left\{ u,v\right\} $. Assume, for simplicity, that
$X_{\ell}=\left\{ u_{\ell},v_{\ell}\right\} =\left\{ a,b\right\} $,
and that in the first step the first option in \prettyref{eq:straight-euclides}
holds, so that $X_{\ell-1}=\left\{ a,ba\right\} $. Let $r$ be the
number of times the first option in \prettyref{eq:straight-euclides}
holds before it fails, i.e.\ $X_{\ell-2}=\left\{ a,ba^{2}\right\} ,\ldots,X_{\ell-r}=\left\{ a,ba^{r}\right\} $
(possibly $r=1$). If $r=\ell$ then the statement holds. Otherwise,
$X_{\ell-r-1}=\left\{ aba^{r},ba^{r}\right\} $, and now every cyclic
word which is a product of the elements of $X_{\ell-r-1}$ (with positive
exponents only) clearly satisfies the statement of the corollary with
$\varepsilon=1$ and $m=r$. Since $u$ and $v$ are such words, we
are done. 
\end{proof}

\section{The counterexample\label{sec:The-counterexample}}

Let $X=\left\{ a,b\right\} $ and $\mathbf{F}_{2}=\mathbf{F}\left(X\right)$.
In this section we prove the following:
\begin{prop}
\label{prop:counterexample}Let $H=\left\langle a^{2}b^{2}\right\rangle $,
and $J=\left\langle a^{2}b^{2},ab\right\rangle $. Then $H\leq_{{\scriptscriptstyle \overset{\twoheadrightarrow}{Y}}}J$
for every basis $Y$ of $\mathbf{F}_{2}$, but $J$ is not an algebraic
extension of $H$.\end{prop}
\begin{proof}
First, as $H$ is a free factor of $J$ (since $J=H*\left\langle ab\right\rangle $),
it is clear that $J$ is not an algebraic extension of $H$, and it
is left to show that $H$ covers $J$ with respect to every basis
$Y=\left\{ u,v\right\} $. For any automorphism $\varphi$ of $\mathbf{F}_{2}$,
$H\covers J$ iff $\varphi\left(H\right)$ covers $\varphi\left(J\right)$
w.r.t.\ the basis $\varphi\left(X\right)=\left\{ \varphi\left(a\right),\varphi\left(b\right)\right\} $.
As $\varphi\left(X\right)$ achieves all bases of $\mathbf{F}_{2}$,
what we seek to show is equivalent to the assertion that $\left\langle u^{2}v^{2}\right\rangle \covers\left\langle u^{2}v^{2},uv\right\rangle $
for every basis $\left\{ u,v\right\} $. 

By Lemma \ref{lem:covers-iff-appears}, showing that $\left\langle u^{2}v^{2}\right\rangle $
$X$-covers $\left\langle u^{2}v^{2},uv\right\rangle $ is equivalent
to verifying that $uv$ appears in $\Gamma_{X}\left(\left\langle u^{2}v^{2}\right\rangle \right)$.
For the case where $u$ and $v$ are CR this is shown in Lemma \ref{lem:both-CR},
and the case where only one of them is CR is handled in Lemma \ref{lem:one-CR}.
For the general case, let $Y=\left\{ \overline{u},\overline{v}\right\} $
be the base at hand, and write $\overline{u}=xux^{-1}$ and $\overline{v}=yvy^{-1}$
with $u,v$ CR. By Lemma \ref{lem:general-bases}, $x$ is a prefix
of $y$, or vice-versa. Thus, if $w$ is the shorter among $x,y$
then $w^{-1}Yw$ is a basis with one CR element, which was already
handled. Inferring from this the result for the original $Y$ is done
in Lemma \ref{lem:none-CR}. For this we need an additional technical
assumption on $\Gamma_{X}\left(\left\langle w^{-1}\overline{u}^{2}\overline{v}^{2}w\right\rangle \right)$,
which is seen to hold in Lemmas \ref{lem:both-CR} and \ref{lem:one-CR}.\end{proof}
\begin{lem}
\label{lem:none-CR}Let $\left\{ \overline{u},\overline{v}\right\} $
be a basis of $\mathbf{F}_{2}$ such that $\overline{u}$ and $\overline{v}$
share a common prefix $w$ and a common suffix $w^{-1}$, and write
$\overline{u}=wuw^{-1}$ and $\overline{v}=wvw^{-1}$. If
\begin{enumerate}
\item $uv$ appears in $\Gamma_{X}\left(\left\langle u^{2}v^{2}\right\rangle \right)$,
and
\item either $u_{1}$ or $\left(v^{-1}\right)_{1}$ emanates from the basepoint
of $\Gamma_{X}\left(\left\langle u^{2}v^{2}\right\rangle \right)$,
\end{enumerate}

then \textup{$\overline{u}\overline{v}$ appears in} $\Gamma_{X}\left(\left\langle \overline{u}^{2}\overline{v}^{2}\right\rangle \right)$.

\end{lem}
\begin{proof}
Observe the graph $\Gamma$, which is obtained by attaching a path
labeled by $w$ to the basepoint of $\Gamma_{X}\left(\left\langle u^{2}v^{2}\right\rangle \right)$
and moving the basepoint to the origin of the $w$-path:
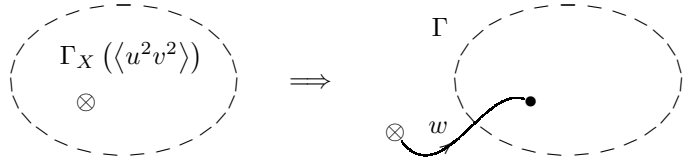
\[\xy(0,0)*{\otimes}="bp"+(6,6)*{\Gamma_{X}\left(\left\langle u^{2}v^{2}\right\rangle \right)};	 (5,3)*\xycircle(15,10){--};  \endxy \qquad\raisebox{8pt}{$\Longrightarrow$}\qquad
\xy(0,0)*{\bullet}="bp"+(-12,10)*{\Gamma};
(5,3)*\xycircle(15,10){--};
(-18,-4)*{\otimes}="s1";
"s1";"bp" **\crv{ (-12,-14) & (-6,4)} ?(.43)*\dir{>} +(-2,2)*{w};
\endxy\]The graph $\Gamma$ folds into $\Gamma_{X}\left(\pi_{1}^{X}\left(\Gamma\right)\right)=\Gamma_{X}\left(\left\langle \overline{u}^{2}\overline{v}^{2}\right\rangle \right)$,
since if satisfies the conditions of Lemma \ref{lem:no-trim}: the
only vertex that needs checking is the gluing place, and there the
conditions hold by assumption \emph{(2)} and the fact that $w^{-1}\perp u,v^{-1}$
(as $\overline{u},\overline{v}$ are reduced). Finally, since $uv$
appears in $\Gamma_{X}\left(\left\langle u^{2}v^{2}\right\rangle \right)$,
$\overline{uv}=wuvw^{-1}$ appears in $\Gamma$, and thus also in
its folding $\Gamma_{X}\left(\left\langle \overline{u}^{2}\overline{v}^{2}\right\rangle \right)$.\end{proof}
\begin{lem}
\label{lem:both-CR}If $\left\{ u,v\right\} $ is a CR basis of $\mathbf{F}_{2}$
then
\begin{enumerate}
\item $uv$ appears in $\Gamma_{X}\left(\left\langle u^{2}v^{2}\right\rangle \right)$,
and
\item $u_{1}$ or $\left(v^{-1}\right)_{1}$ emanates from the basepoint
of $\Gamma_{X}\left(\left\langle u^{2}v^{2}\right\rangle \right)$.
\end{enumerate}
\end{lem}
\begin{proof}
If $\left|u\right|+\left|v\right|=2$ then the claims hold, so assume
that $\left|u\right|+\left|v\right|\geq3$. Since $\Gamma_{X}\left(\left\langle u^{2}v^{2}\right\rangle \right)=\Gamma_{X}\left(\left\langle v^{-2}u^{-2}\right\rangle \right)$
and $\Gamma_{X}\left(\left\langle u^{2}v^{2},uv\right\rangle \right)=\Gamma_{X}\left(\left\langle v^{-2}u^{-2},v^{-1}u^{-1}\right\rangle \right)$,
one can look at $\left\{ v^{-1},u^{-1}\right\} $ instead of $\left\{ u,v\right\} $
(this also does not affect assertion \emph{(2)}), and thus it is enough
to handle the cases where $\left|u\right|<\left|v\right|$.

Observe the graph $\Gamma=\vcenter{\xymatrix@1@R=2pt@C=15pt{ & \bullet\ar[dr]^u\\\otimes\ar[ru]^u &  & \circ\ar[dl]^{v}\\ & \bullet\ar[lu]^{v}}}$,
which obviously satisfies $\pi_{1}^{X}\left(\Gamma\right)=\left\langle u^{2}v^{2}\right\rangle $.
At the black vertices there can be no folding, as $u^{-1}\perp u$
follows from $u$ being CR, and likewise for $v$. In what follows
we will continue to mark by $\bullet$ vertices at which we already
know that no folding can occur, and by $\circ$ vertices at which
we do not know this.

Assume first that $u^{-1}$ is not a prefix of $v$, and $m$ is the
maximal common prefix of $u^{-1}$ and $v$. Writing $u=u'm^{-1}$
and $v=mv'$, $\Gamma$ folds into $\vcenter{\xymatrix@1@R=2pt@C=15pt{ & \bullet\ar[dr]^{u'}\\\otimes\ar[ru]^{u} &  & \bullet\ar[dl]^{v'} & \bullet\ar[l]_{\:m}\\ & \bullet\ar[lu]^{v}}}$.
After trimming one obtains $\Gamma'=\vcenter{\xymatrix@1@R=2pt@C=15pt{ & \bullet\ar[dr]^{u'}\\\otimes\ar[ru]^{u} &  & \bullet\ar[dl]^{v'}\\ & \bullet\ar[lu]^{v}}}$,
which satisfies the conditions of Lemma \ref{lem:no-trim}: $u^{-1}\perp u'$
since $u$ is CR and $u'$ is a prefix of $u$, and likewise for $v$
and $v'^{-1}$; $u'^{-1}\perp v'$ by the maximality of $m$. Therefore,
$\Gamma'$ folds into $\Gamma_{X}\left(\left\langle u^{2}v^{2}\right\rangle \right)$.
Since $uv$ appears in $\Gamma'$, and both $u_{1}$ and $\left(v^{-1}\right)_{1}$
leave its basepoint, the same holds after any foldings, and in particular
in $\Gamma_{X}\left(\left\langle u^{2}v^{2}\right\rangle \right)$.

Assume now that $u^{-1}$ is a prefix of $v$ and write $v=u^{-1}t$.
Now $\Gamma$ folds and trims into $\vcenter{\xymatrix@1@R=2pt@C=15pt{ & \circ\ar[dr]^{t}\\\otimes\ar[ru]^{u} &  & \bullet\\ & \bullet\ar[lu]^{t}\ar[ru]_{u}}}$,
as $u\perp t$ and $t^{-1}\perp u^{-1}$ follow from $v=u^{-1}t$
being CR. Let $m$ be the maximal common prefix of $u^{-1}$ and $tu^{-1}$,
and write $u=\overline{u}m^{-1}$, $tu^{-1}=mq$ (note that $\left|tu^{-1}\right|>\left|u^{-1}\right|\geq\left|m\right|$).
The last graph then folds and trims into $\Gamma'=\vcenter{\xymatrix@1@R=1.5pt@C=15pt{ & \bullet\ar[dd]_{q}\\\otimes\ar[ru]^{\overline{u}}\\ & \bullet\ar[lu]^{t}}}$
(with $\overline{u}$ possibly empty, in which case $\Gamma'=\xymatrix@1@C=18pt{\otimes\ar@/^{0.7pc}/[r]_{q}&\bullet\ar@/^{0.7pc}/[l]_{t}}$).
This graph satisfies the conditions of Lemma \ref{lem:no-trim}: $q^{-1}\perp t$
since $tu^{-1}=mq$ is CR, being a cyclic rotation of $v$, and if
$\overline{u}$ is not empty then $\overline{u}^{-1}\perp q$ by maximality
of $m$. Since $uv=t$ appears in $\Gamma'$ and $\left(t^{-1}\right)_{1}=\left(v^{-1}\right)_{1}$
leaves its basepoint, the same holds for $\Gamma_{X}\left(\left\langle u^{2}v^{2}\right\rangle \right)$,
which is obtained from it by foldings.\end{proof}
\begin{lem}
\label{lem:one-CR}If $\left\{ u,v\right\} $ is a basis of $\mathbf{F}_{2}$
with $u$ or $v$ CR then 
\begin{enumerate}
\item $uv$ appears in $\Gamma_{X}\left(\left\langle u^{2}v^{2}\right\rangle \right)$,
and
\item $u_{1}$ or $\left(v^{-1}\right)_{1}$ emanates from the basepoint
of $\Gamma_{X}\left(\left\langle u^{2}v^{2}\right\rangle \right)$.
\end{enumerate}
\end{lem}
\begin{proof}
If both $u$ and $v$ are CR then we are done by Lemma \ref{lem:both-CR}.
Again, by replacing $u$ and $v$ with $v^{-1}$ and $u^{-1}$ respectively
we can assume that $u$ is CR and $v$ is not (here $u$ is not necessarily
shorter). Writing $v=w\overline{v}w^{-1}$ with $\overline{v}$ CR
(and $w\neq1$), $w$ is a prefix of some power of $u$ or $u^{-1}$
by Lemma \ref{lem:general-bases}. The graph formed by a single $u^{2}v^{2}$-loop
folds and trims into $\Gamma=\vcenter{\xymatrix@1@R=2pt@C=15pt{ & \bullet\ar[r]_{u} & \circ\ar[dr]^{w}\\\otimes\ar[ru]^{u}\ar[dr]_{w} &  &  & \bullet\ar[dl]^{\overline{v}}\\ & \bullet & \bullet\ar[l]_{\overline{v}}}}$,
where at all the black vertices there can be no folding since $u$
and $\overline{v}$ are CR and $v$ is reduced. 

If $w$ is a prefix of a (positive) power of $u$ then at $\circ$
there is no folding as well since $u$ is CR, so that $\Gamma_{X}\left(\left\langle u^{2}v^{2}\right\rangle \right)$
is obtained from $\Gamma$ by foldings. Therefore to establish that
$uv$ appears in $\Gamma_{X}\left(\left\langle u^{2}v^{2}\right\rangle \right)$
it is enough to show that it appears in $\Gamma$. This is not obvious
in first sight, but it is true: $w$ is a prefix of a positive power
of $u$, so that $uw$ is a prefix of $u^{2}w$, hence $uv=uw\overline{v}w^{-1}$
indeed appears in $\Gamma$. Finally, both $u_{1}$ and $\left(v^{-1}\right)_{1}=w_{1}$
leave the basepoint of $\Gamma$ and thus also of $\Gamma_{X}\left(\left\langle u^{2}v^{2}\right\rangle \right)$.

We assume now that $w$ is a prefix of some power of $u^{-1}$, and
observe six cases. 

\textbf{\emph{Case}}\textbf{ }\textbf{\emph{$\boldsymbol{\left(i\right)}$}}:
$2\left|u\right|\leq\left|w\right|$, so that $w=u^{-2}\overline{w}$
with $\overline{w}$ possibly empty. In this case $\Gamma$ folds
and trims into $\Gamma'=\vcenter{\xymatrix@1@R=2pt@C=15pt{ & \bullet\ar[r]_{\overline{v}} & \bullet\ar[dr]^{\overline{v}}\\\otimes\ar[ru]^{\overline{w}} &  &  & \bullet\\ & \bullet\ar[ul]^{u} & \bullet\ar[ur]_{\overline{w}}\ar[l]_{u}}}$,
which satisfies Lemma \ref{lem:no-trim} (even if $\overline{w}=1$).
Now $uv=u^{-1}\overline{w}\overline{v}\overline{w}^{-1}u^{2}$ appears
in $\Gamma'$: $\overline{v}\overline{w}^{-1}u^{2}$ is a suffix of
$\Gamma'$ oriented clockwise, and $u^{-1}\overline{w}$ is a prefix
of $\Gamma'$ oriented counterclockwise since $\overline{w}$ is a
prefix of $u^{-1}\overline{w}$. In addition, $\left(v^{-1}\right)_{1}=w_{1}=\left(u^{-1}\right)_{1}$
leaves $\otimes$.

\textbf{\emph{Case $\boldsymbol{\left(ii\right)}$}}: $\left|u\right|<\left|w\right|<2\left|u\right|$,
so that we can write $u=qr$ and $w=r^{-1}q^{-1}r^{-1}$ with $q,r\neq1$.
Now $\Gamma$ folds and trims into $\Gamma'=\vcenter{\xymatrix@1@R=2pt@C=15pt{ & \circ\ar[r]_{\overline{v}} & \bullet\ar[dr]^{\overline{v}}\\\otimes\ar[ru]^{q} &  &  & \bullet\ar[dl]^{r}\\ & \bullet\ar[ul]^{r} & \bullet\ar[l]_{q}}}$,
and no folding can occur at the black vertices. If at $\circ$ there
is no folding as well, then $uv=r^{-1}\overline{v}rqr$ appears in
$\Gamma'$ and $\left(v^{-1}\right)_{1}=w_{1}=\left(r^{-1}\right)_{1}$
leaves $\otimes$, yielding the same for $\Gamma_{X}\left(\left\langle u^{2}v^{2}\right\rangle \right)$.
\\
Assume now that there is folding at $\circ$, so that $q^{-1}$ and
$\overline{v}^{2}$ have a common prefix. If this prefix is shorter
than $\left|\overline{v}\right|$ than the $\overline{v}rqr$ part
in $\Gamma'$ survives in $\Gamma_{X}\left(\left\langle u^{2}v^{2}\right\rangle \right)$,
and thus $uv=r^{-1}\overline{v}rqr$ still appears in it. Otherwise,
$\overline{v}$ is a prefix of $q^{-1}$, so that $r^{-1}\overline{v}$
is a prefix of $\Gamma'$ oriented CCW, and $rqr$ is a suffix of
$\Gamma'$ oriented CW, so that $uv$ appears already in the lower
half of $\Gamma'$. Finally, this half survives in $\Gamma_{X}\left(\left\langle u^{2}v^{2}\right\rangle \right)$
since $q^{-1}$ cannot overlap with $r$, since $u=qr$ is monotone
by Proposition \ref{prop:monotone}. In both cases $\left(v^{-1}\right)_{1}=\left(r^{-1}\right)_{1}$
still leaves $\otimes$.

\textbf{\emph{Case }}\emph{$\boldsymbol{\left(iii\right)}$}: $w=u^{-1}$.
The reasoning here is as in the previous case with $r=1$.

\medskip{}

In cases $\boldsymbol{\left(iv\right)-\left(vi\right)}$ $w$ is a
proper prefix of $u^{-1}$, and we write $u=qr$ and $w=r^{-1}$ (with
$r,q\neq1$). Here $\Gamma$ folds and trims into $\Gamma'=\vcenter{\xymatrix@1@R=2pt@C=15pt{
                   & \bullet\ar[r]_{q} & \circ\ar[dr]^{\overline{v}}\\
\bullet\ar[ru]^{r} &                            &                          & \bullet\ar[dl]^{\overline{v}}\\
                   & \otimes\ar[ul]^{q}         & \filleddiamond\ar[l]_{\:r}}}$, with folding possible only at $\circ$, and the folding cannot reach
past $\filleddiamond$ due to the monotonicity of $u=qr$. This already
shows that $\left(v^{-1}\right)_{1}=\left(r^{-1}\right)_{1}$ must
leave the basepoint of the final Stallings graph $\Gamma_{X}\left(\left\langle u^{2}v^{2}\right\rangle \right)$.
It is left to show that the folding and trimming at $\circ$ does
not prevent $uv=q\overline{v}r$ from appearing in $\Gamma_{X}\left(\left\langle u^{2}v^{2}\right\rangle \right)$.
If the lower half of $\Gamma'$ survives the folding and trimming
then $uv$ certainly appears in it. We assume therefore that there
is folding at $\circ$, and that it encompasses either all of $\overline{v}$
to the right of $\circ$ or all of $rq$ to the left of it (i.e.\ it
reaches the lower half of $\Gamma'$). 

\textbf{\emph{Case}}\textbf{ }\textbf{\emph{$\boldsymbol{\left(iv\right)}$}}:
$\left|\overline{v}\right|\leq\left|q\right|$. By our assumption,
$\overline{v}$ is a prefix of $q^{-1}$, so $q=\overline{q}\overline{v}^{-1}$
($\overline{q}$ may be empty). $\Gamma'$ then folds and trims into
$\Gamma''=\vcenter{\xymatrix@1@R=2pt@C=15pt{
                   & \bullet\ar[dl]_{\overline{v}}\ar[r]_{r} & \bullet\ar[dr]^{\overline{q}}\\
\bullet &                            &                          & \circ\ar[dl]^{\overline{v}}\\
                   & \otimes\ar[ul]^{\overline{q}}         & \filleddiamond\ar[l]_{\:r}}}$. Now $uv=q\overline{v}r=\overline{q}r$ and since the folding from
$\circ$ downward must stop at $\filleddiamond$ (or earlier), the
$\xymatrix@1@R=2pt@C=15pt{
\bullet & \otimes\ar[l]_{\:\overline{q}}         & \filleddiamond\ar[l]_{\:r}}$ part survives in $\Gamma_{X}\left(\left\langle u^{2}v^{2}\right\rangle \right)$
(since $\left|\overline{v}\right|<\left|\overline{q}^{-1}r^{-1}\overline{v}\right|$)
and we are done.

\textbf{\emph{Case}} \emph{$\boldsymbol{\left(v\right)}$}: $\left|q\right|<\left|\overline{v}\right|\leq\left|rq\right|$.
Now we can assume that $\overline{v}$ is a prefix of $q^{-1}r^{-1}$,
so that $r=st$ and $\overline{v}=q^{-1}t^{-1}$ (possibly with $s=1$).
In this case $uv=q\overline{v}r=t^{-1}r$ already appears in the $\xymatrix@1@R=2pt@C=15pt{\otimes         & \filleddiamond\ar[l]_{\:r}}$
part of $\Gamma'$, which always survives due to monotonicity.

\textbf{\emph{Case}}\emph{ $\boldsymbol{\left(vi\right)}$}: $\left|rq\right|<\left|\overline{v}\right|$.
Now we can assume that $q^{-1}r^{-1}$ is a prefix of $\overline{v}$.
Therefore, $uv=q\overline{v}r$ is a suffix of $\overline{v}r$, and
thus appears in the $\xymatrix@1@R=2pt@C=15pt{
\otimes         & \filleddiamond\ar[l]_{\:r} & \bullet\ar[l]_{\overline{v}}}$ part in $\Gamma'$. If $\overline{v}$ is not a prefix of $q^{-1}r^{-1}q^{-1}$
then the folding from $\circ$ downward stops before reaching this
part, and we are done. We thus add the assumption that $\overline{v}$
is a prefix of $q^{-1}r^{-1}q^{-1}$. Since $\left|rq\right|<\left|\overline{v}\right|$
we can write $q=st$ so that $\overline{v}=q^{-1}r^{-1}t^{-1}=t^{-1}s^{-1}r^{-1}t^{-1}$.
Now $\Gamma'$ folds and trims into $\Gamma''=\vcenter{\xymatrix@1@R=2pt@C=15pt{
                   & \bullet\ar[dl]_{t} & \bullet\ar[l]^{s}\\
\circ &                            &                          & \bullet\ar[ul]_{r}\\
                   & \otimes\ar[ul]^{s}         & \filleddiamond\ar[l]_{\:r}\ar[ur]_t}}$, and $uv=r^{-1}t^{-1}r$ appears in $\xymatrix@1@R=2pt@C=15pt{
\otimes         & \filleddiamond\ar[l]_{\:r}\ar[r]^t & \bullet}$, which survives any further folding since $\left|s\right|<\left|t^{-1}s^{-1}r^{-1}\right|$.
\end{proof}

\section{\label{sec:Epilogue}Epilogue}

While the original conjecture \cite[§5(1)]{MVW07} fails, it is plausible
that some modification of it holds. One possible option is the following:
\begin{conjecture}
\label{con:new-con}Let $H\leq J$ be subgroups of the free group
$\mathbf{F}$. Then $H\leq_{alg}J$ iff $H\covers J$ for every free
extension $\mathbf{F}'$ of $\mathbf{F}$, and every basis $X$ of
$\mathbf{F}'$.
\end{conjecture}
Since the relation $H\leq_{alg}J$ does not depend on the ambient
group, one direction holds as before. But in contrast with the original
conjecture, the example in Section \ref{sec:The-counterexample} is
no longer a counterexample: let $\mathbf{F}=\mathbf{F}\left(a,b\right)$,
$H=\left\langle a^{2}b^{2}\right\rangle $ and $J=\left\langle ab,a^{2}b^{2}\right\rangle $.
For $\mathbf{F}'=\mathbf{F}\left(a,b,c\right)$ and $X=\left\{ a,cb^{-1},cbc^{-1}\right\} $,
$H$ does not $X$-cover $J$: denote $x=a$, $y=cb^{-1}$ and $z=cbc^{-1}$.
Then written in this basis, $H=\left\langle x^{2}y^{-1}z^{2}y\right\rangle $
and $J=\left\langle x^{2}y^{-1}z^{2}y,xy^{-1}zy\right\rangle $. By
Lemma \ref{lem:covers-iff-appears}, $H\covers J$ iff $xy^{-1}zy$
appears in $\Gamma_{\left\{ x,y,z\right\} }\left(\left\langle x^{2}y^{-1}z^{2}y\right\rangle \right)$,
which is not the case.

Another plausible option is that the original conjecture from \cite[§5(1)]{MVW07}
holds for free groups of rank three or more, as it is clear that the
counterexample exploits many idiosyncrasies of $\mathbf{F}_{2}$.
If this is true, then Conjecture \ref{con:new-con} follows as well.

\bibliographystyle{amsalpha}
\bibliography{PrimitiveBib}

\end{document}